\documentclass{amsart}
\usepackage{amsmath, amstext, amsbsy, amssymb}
\usepackage{graphicx}
\usepackage{subfigure}
\usepackage{picinpar}
\usepackage{psfrag}
\input xy \xyoption{all}

\hoffset \voffset \oddsidemargin=55pt \evensidemargin=55pt
\numberwithin{equation}{section}
\def\beq{\begin{eqnarray}}
\def\eeq{\end{eqnarray}}
\def\beqs{\begin{eqnarray*}}
\def\eeqs{\end{eqnarray*}}

\def\mz{{\mathbb Z}}
\def\mc{{\mathbb C}}

\def\a{\alpha}

\newfont{\df}{eufm10}

\def\a{\alpha}

\def\mc{{\mathbb C}}

\def\id{\hbox{\rm id}}

\def\sg{\frak{g}}

\title[]
{Heisenberg double of the generalized quantum euclidean group and its representations}

\author[L. Xia]{Limeng Xia}
\address{XIA: Institute of Applied System Analysis, Jiangsu University, XueFu Road 301,
Zhenjiang 212013, Jiangsu,  PR China} \email{xialimeng@ujs.edu.cn}

\date{}
\begin{document}
\maketitle
\def\oh{{\mathcal{H}_q}}
\def\oq{{\mathcal{O}_q}}
\def\ou{{\mathcal{U}_q}}
\def\od{{\mathcal{D}_q}}
\def\oqmn{{\mathcal{O}_q(E_2;m,n)}}
\def\odmn{{\mathcal{D}_q(E_2;m,n)}}
\def\oumn{{\mathcal{U}_q(\frak{e}_2;m,n)}}
\def\mk{{\mathbb{K}}}

\newtheorem{theo}{Theorem}[section]
\newtheorem{theorem}[theo]{Theorem}
\newtheorem{defi}[theo]{Definition}
\newtheorem{lemma}[theo]{Lemma}
\newtheorem{coro}[theo]{Corollary}
\newtheorem{prop}[theo]{Proposition}
\newtheorem{remark}[theo]{Remark}

\begin{abstract} The generalized quantum Euclidean group $\oq(\frak{b}_{m,n})$ is a natural generalization of the  quantum Euclidean group $\oq(\frak{b}_{1,1})$. The  Heisenberg double  $\od(\frak{b}_{m,n})$ of  $\oq(\frak{b}_{m,n})$ is the smash product of $\oq(\frak{b}_{m,n})$ with its Hopf dual $\ou(\frak{b}_{m,n})$.  In this paper, we study the weight modules, the prime spectrum  and the automorphism group of the Heisenberg double $\od(\frak{b}_{m,n})$.
\end{abstract}



\smallskip\bigskip

\section{Introduction}

Throughout this paper, $\mk$ is an algebraically closed field of  characteristic zero, for example $\mc$. Let $\mk^*=\mk\setminus\{0\}$ and $q\in\mk^*$ is not a root of unity.

Assume that $\sg$ is an indecomposable Lie algebra of dimension $3$, then it is one of the following:
\begin{itemize}
\item{The three dimensional simple Lie algebra $\frak{sl}_2$.}
\item{The three dimensional Heisenberg Lie algebra.}
\item{$\frak{b}_\lambda={\rm span}\{h, x, y\}$ such that $[h,x]=x, [h,y]=\lambda y, [x,y]=0$ for some $\lambda\in\mk^*$.}
\end{itemize}
They are very important in the representation theory of Lie algebras. The simple modules over $\frak{sl}_2$ and the simple modules over Heisenberg Lie algebra have been completely classified in \cite{B}.
If $\lambda$ is not rational,  the classification of simple $\frak{b}_\lambda$-modules is not yet done.  When $\lambda$ is rational, $\frak{b}_\lambda$ is isomorphic to some $\frak{b}_{m,n}$, which is the three dimensional indecomposable Lie algebra with basis $\{h,x,y\}$ and brackets
\beqs [h,x]=mx, [h,y]=ny, [x,y]=0,\eeqs
for some nonzero integers $m,n$. The simple modules over $\frak{b}_{m,n}$ have been classified, and the classification has a little difference according to ${\rm sgn}(mn)$.

In this paper, we consider the Heisenberg double of quantum groups obtained from the quantized universal enveloping algebra of $\frak{b}_{m,n}$.

\begin{defi}The Hopf algebra $\oq(\frak{b}_{m,n})$ is generated by the elements $a,a^{-1},b,c$ subject to the following relations
\beqs ab=q^nba,\;ac=q^{m}ca,\;bc=cb,\;aa^{-1}=a^{-1}a=1.\eeqs
Its comultiplication $\Delta$, counit $\varepsilon$, antipode $S$ are defined by
\beqs &&\Delta(a)=a\otimes a,\; \varepsilon(a)=1,\; S(a)=a^{-1},\\
&&\Delta(b)=b\otimes a^{-n}+a^n\otimes b,\; \varepsilon(b)=0,\; S(b)=-q^{-n^2}b,\\
&&\Delta(c)=c\otimes a^m+a^{-m}\otimes c,\; \varepsilon(c)=0,\; S(c)=-q^{m^2}c.\eeqs
\end{defi}

\begin{defi}The Hopf algebra $\ou(\frak{b}_{m,n})$ is generated by the elements $K,K^{-1},E,F$ subject to the following relations
\beqs KE=q^{2m}EK,\;KF=q^{-2n}FK,\;EF=FE,\;KK^{-1}=K^{-1}K=1.\eeqs
Its comultiplication $\Delta$, counit $\varepsilon$, antipode $S$ are defined by
\beqs &&\Delta(K)=K\otimes K,\; \varepsilon(K)=1,\; S(K)=K^{-1},\\
&&\Delta(E)=E\otimes K^{m}+1\otimes E,\; \varepsilon(E)=0,\; S(E)=-EK^{-m},\\
&&\Delta(F)=F\otimes 1+K^{-n}\otimes F,\; \varepsilon(F)=0,\; S(F)=-K^{n}F.\eeqs
\end{defi}

Note that $\oq(\frak{b}_{m,n})$ is a quantized analogue of $U(\frak{b}_{m,n})$, at the same time  $\ou(\frak{b}_{m,n})$ is a quantized analogue of $U(\frak{b}_{m,-n})$. In particular, $\oq(\frak{b}_{1,1})$ is the quantum Euclidean group and $\frak{b}_{1,-1}$ is the Euclidean Lie algebra. Thus we call $\oq(\frak{b}_{m,n})$ the generalized quantum Euclidean group and call $\frak{b}_{m,-n}$ the generalized Euclidean Lie algebra.

\begin{prop}$\ou(\frak{b}_{m,-n})$ is isomorphic to a Hopf subalgebra of $\oq(\frak{b}_{m,n})$.
\end{prop}
\begin{proof}It is easy to verify that the following map defines an embedding of Hopf algebras:
\beqs K\mapsto a^2,\; E\mapsto a^mc,\; F\mapsto ba^{n}.\eeqs
\end{proof}

There is a unique non-degenerate dual pair $\langle\cdot,\cdot\rangle:\ou(\frak{b}_{m,n})\times\oq(\frak{b}_{m,n})\rightarrow\mk$ determined by non-vanishing pairs:
\beqs \langle K, a\rangle={q}^{-1}, \langle K, a^{-1}\rangle=q, \langle E,c\rangle=1,\langle F,b\rangle=1.\eeqs

{\it Note that there are two different non-degenerate dual pairs $\langle\cdot,\cdot\rangle:\oq(\frak{b}_{m,-n})\times\oq(\frak{b}_{m,n})\rightarrow\mk$.
}

This dual pair makes $\oq(\frak{b}_{m,n})$ to be a left $\ou(\frak{b}_{m,n})$-module algebra with the action defined by
\beqs u\cdot x=\sum_{x}x_{(1)}\langle u,x_{(2)}\rangle,\;\forall u\in\ou(\frak{b}_{m,n}), x\in\oq(\frak{b}_{m,n}).\eeqs
Explicitly, the action takes the following form
\beqs &&K\cdot a=q^{-1}a,\;K\cdot b=q^nb,\;K\cdot c=q^{-m}c,\\
&&E\cdot a=0,\;E\cdot b=0,\;E\cdot c=a^{-m},\\
&&F\cdot a=0,\;F\cdot b=a^n,\;F\cdot c=0.\eeqs

We define the {\it Heisenberg double of the generalized quantum Euclidean group}, which is the smash product algebra
\beqs \od(\frak{b}_{m,n})=\oq(\frak{b}_{m,n})\rtimes\ou(\frak{b}_{m,n}).\eeqs

In this paper we are concerned with the Heisenberg double $\od(\frak{b}_{m,n})$ and its representation theory.

\begin{defi}The Heisenberg double $\od(\frak{b}_{m,n})$ of the generalized quantum Euclidean group is the unital algebra generated by the elements $a, a^{-1}, b, c, K, E, F$ subject to the following relations
 \beqs&&ab=q^nba,\;ac=q^{m}ca,\;bc=cb,\;aa^{-1}=a^{-1}a=1,\\
 &&KE=q^{2m}EK,\;KF=q^{-2n}FK,\;EF=FE,\;KK^{-1}=K^{-1}K=1,\\
 &&Ka=q^{-1}aK,\;Kb=q^nbK,\;Kc=q^{-m}cK,\\
 &&Ea=aE,\;Eb=bE,\;Ec=cE+a^{-m}K^m,\\
 &&Fa=q^naF,\;Fb=q^{-n^2}bF+a^n,\;Fc=q^{mn}cF.\eeqs
\end{defi}
The centrally extended Heisenberg double of $SL_2$ and the Heisenberg double $\od(\frak{b}_{1,1})$ of the quantum Euclidean group were studied by Tao (see \cite{Tao, Tao2}).  For the general construction of the Heisenberg double of a Hopf algebra, one is referred to \cite{Lu}.

\section{The weight modules}

Because the map $\xi_K: x\rightarrow KxK^{-1}$ (resp. $\xi_a: x\rightarrow axa^{-1}$) is diagonal, it is possible to study the $K$-weight (resp. $a$-weight) modules over $\od(\frak{b}_{m,n})$.

\begin{defi}A $\od(\frak{b}_{m,n})$-module $V$ is called a $K$-weight module provided that $V=\oplus_{\lambda\in\mk^*}V_\lambda$ where $V_\lambda=\{v\in V:Kv=\lambda v\}$. Denote by ${\rm supp}_K(V)=\{\lambda\in\mk^*|V_\lambda\not=0\}$ the set of all $K$-weights of $V$.
\end{defi}

\begin{defi}A $\od(\frak{b}_{m,n})$-module $V$ is called an $a$-weight module provided that $V=\oplus_{\mu\in\mk^*}V^\mu$ where $V^\mu=\{v\in V:av=\mu v\}$. Denote by ${\rm supp}_a(V)=\{\mu\in\mk^*|V^\mu\not=0\}$ the set of all $a$-weights of $V$.
\end{defi}

Let $\od(\frak{b}_{m,n})^+$ be the polynomial algebra $\mk[E,b]$, let $\od(\frak{b}_{m,n})^0$ be the quantum torus $\mk[K^{\pm1},a^{\pm1}]$ and let $\od(\frak{b}_{m,n})^-$ be the quantum polynomial algebra $\mk[F,c]$. Then as algebra, $\od(\frak{b}_{m,n})$ has a triangular decomposition
\beqs \od(\frak{b}_{m,n})=\od(\frak{b}_{m,n})^-\od(\frak{b}_{m,n})^0\od(\frak{b}_{m,n})^+.\eeqs

Let $b'=a^nbK^{-n}, c'=a^{-m}cK^{-m}, E'=a^{2m}E, F'=q^{-n^2}a^{-2n}FK^n$. Then $\{a^{\pm1}, K^{\pm1}, b',c',E',F'\}$ is also a generating set of $\od(\frak{b}_{m,n})$. In particular, we have
\beqs &&Kb'=b'K,\; Kc'=c'K,\; KE'=E'K,\; KF'=F'K,\\
&&ab'=b'a,\; ac'=c'a,\; aE'=E'a,\; aF'=F'a,\\
&&b'c'=q^{2mn}c'b', \; E'b'=q^{2mn}b'E',\; F'b'=q^{-2n^2}b'F'+1, \\
&&E'c'=q^{2m^2}c'E'+1,\; F'c'=q^{-2mn}c'F',\;E'F'=q^{-2mn}F'E'. \eeqs

Let $S$ be the subalgebra $\mk[b',c',E',F']$, then we have the following isomorphism of algebras
\beqs \od(\frak{b}_{m,n})&\cong&\od(\frak{b}_{m,n})^0\otimes S.\eeqs

\begin{lemma}\label{S-mod}The $S$ has no finite dimensional modules.
\end{lemma}
\begin{proof}Note that $S$ has two subalgebras $S_1=\mk[E',c']$ and $S_2=\mk[F',b']$, both of them are generalized Weyl algebras. It is known that a generalized Weyl algebra has no finite dimensional modules.
\end{proof}

Our first main result is  the following theorem.
\begin{theo}\label{weight} (i) Let $V$ be a simple $K$-weight $\od(\frak{b}_{m,n})$-module. Then ${\rm supp}_K(V)=\{q^i\lambda|i\in\mz\}$ for some $\lambda$ and any $K$-weight space of $V$ is a simple $S$-module. In particular, any $K$-weight space of $V$ is infinite dimensional.

(ii)  Let $W$ be a simple $a$-weight $\od(\frak{b}_{m,n})$-module. Then ${\rm supp}_a(W)=\{q^i\mu|i\in\mz\}$ for some $\mu$ and any $a$-weight space of $W$ is a simple $S$-module. In particular, any $a$-weight space of $W$ is infinite dimensional.
\end{theo}

%
%
%
%
%
%
%
%
%
%


\begin{proof} We only prove part (i), the proof for part (ii) is very similar.

Let $V$ be a simple $K$-weight module with a weight $\lambda\in{\rm supp}_K(V)$. Because $a, a^{-1}$ are invertible and $KaK^{-1}=q^{-1}a$, we have
\beqs aV_\mu\subseteq V_{q^{-1}\lambda},\;a^{-1}V_\mu\subseteq V_{q\lambda}.\eeqs
By induction, it holds ${\rm supp}_K(V)\supseteq\{q^i\lambda|i\in\mz\}$. Because $[K,S]=0$, it implies
\beqs S\cdot V_{q^i\lambda}\subseteq V_{q^i\lambda}.\eeqs
So ${\rm supp}_K(V)=\{q^i\lambda|i\in\mz\}$ and any weight space is an $S$-module. If $V_{q^{i_0}\lambda}$ contains a proper submodule $V'$ over $S$, then
\beqs \sum_{i\in\mz}a^{i}V'\eeqs
is a proper $K$-weight submodule over $\od(\frak{b}_{m,n})$. Thus any weight space is a simple $S$-module. By Lemma \ref{S-mod}, any $K$-weight space of $V$ is infinite dimensional.
\end{proof}

\section{Some simple $S$-modules}

In this section, we construct some classes of simple $S$-modules. In particular, all of these modules have GKdim $2$. 

\begin{prop}
Let $J_1(\sigma,\tau)$ be the left $S$-ideal generated by $b'-\sigma, c'-\tau$ for all $\sigma,\tau\in\mk$ such that $\sigma\tau=0$. Then $S/J_1(\sigma,\tau)$ is a simple $S$-module of GKdim $2$.
\end{prop}
\begin{proof}
It is obvious that $S/J_1(\sigma,\tau)\cong\mk[E',F']v$ with $b'v=\sigma v, c'v=\tau v$.

If $\sigma=0$, for any $x=\sum_{i=0}\varphi_i(E')(F')^i\not=0$, we have
\beqs b'xv=\sum_{i=1}^I\varphi_i(q^{-2mn}E')c_i(F')^{i-1}v\eeqs
for some nonzero constants $c_i, i\geq 1$. Induction on $I$, we may obtain a nonzero element $\varphi(E')v$. It is known that $\mk[E']v$ is a simple $S_1$-module, so this lemma holds.

If $\tau=0$, the proof is similar.
\end{proof}

\begin{prop}
Let $J_2(\sigma,\tau)$ be the left $S$-ideal generated by $b'-\sigma, E'-\tau$ for all $\sigma,\tau\in\mk$ such that $\sigma\tau=0$. Then $S/J_2(\sigma,\tau)$ is a simple $S$-module of GKdim $2$.
\end{prop}
\begin{proof}
First we have $S/J_2(\sigma,\tau)\cong\mk[c',F']v$ with $b'v=\sigma v, E'v=\tau v$.

If $\sigma=0$, for any $x=\sum_{i=0}\varphi_i(c')(F')^i\not=0$, we have
\beqs b'xv=\sum_{i=1}^I\varphi_i(q^{2mn}c')c_i(F')^{i-1}v\eeqs
for some nonzero constants $c_i, i\geq 1$. Induction on $I$, we may obtain a nonzero element $\varphi(c')v$. Because $\mk[c']v$ is also a simple $S_1$-module, so this lemma holds.

If $\tau=0$, the proof is similar.
\end{proof}

Similar to above, we state the following propositions without proof.

\begin{prop}
Let $J_3(\sigma,\tau)$ be the left $S$-ideal generated by $F'-\sigma, c'-\tau$ for all $\sigma,\tau\in\mk$ such that $\sigma\tau=0$. Then $S/J_3(\sigma,\tau)$ is a simple $S$-module of GKdim $2$.
\end{prop}
\begin{prop}
Let $J_4(\sigma,\tau)$ be the left $S$-ideal generated by $F'-\sigma, E'-\tau$ for all $\sigma,\tau\in\mk$ such that $\sigma\tau=0$. Then $S/J_4(\sigma,\tau)$ is a simple $S$-module of GKdim $2$.
\end{prop}

From above propositions, we obtain various simple $K$-weight modules (resp. $a$-weight modules) of GKdim $3$:
\beqs \sum_{i\in\mz}a^iS/J_k(\sigma,\tau),\quad\left(\hbox{resp.}\sum_{i\in\mz}K^iS/J_k(\sigma,\tau)\right)\quad k=1,2,3,4.\eeqs

In particular, when $m=n=1$, the $\od(\frak{b}_{m,n})$-modules
\beqs \sum_{i\in\mz}a^iS/J_k(0,0),\quad\left(\hbox{resp.}\sum_{i\in\mz}K^iS/J_k(0,0)\right)\quad k=1,2,3,4.\eeqs
cover those modules constructed in \cite{Tao2}.

\section{The prime and primitive spectra of $S$}
\def\pp{\frak{p}}
Let $\pp$ be a proper ideal of a ring $R$. Then for all ideals $\frak{i},\frak{j}$, if $\frak{ij}\subseteq\pp$ then either $\frak{i}\subseteq\pp$ or $\frak{j}\subseteq\pp$. In particular, if $R/\pp$ is a domain, the ideal $\pp$ must be prime. Let ${\rm Spec}(S)$ be the set of all prime ideals of $S$.

Define symbols
\beqs\phi_1=E'c'-c'E',\;\phi_2=F'b'-b'F'.\eeqs
\begin{lemma}
\beqs &&\phi_1\phi_2=\phi_2\phi_1,\;\phi_1F'=F'\phi_1,\;\phi_1b'=b'\phi_1,\;E'\phi_2=\phi_2E',\;c'\phi_2=\phi_2c',\\
&&\phi_1E'=q^{-2m^2}E'\phi_1,\;\phi_1c'=q^{2m^2}c'\phi_1,\;F'\phi_2=q^{-2n^2}\phi_2F',\;b'\phi_2=q^{2n^2}\phi_2b'.
\eeqs
\end{lemma}
\begin{proof}It follows by directly computing.
\end{proof}
Assume $d=(m,n)$. Set $(0)=\{0\}, I_1=S\phi_1, I_2=S\phi_2, I_3=I_1+I_2$.
If $mn>0$, we set
\beqs J_1(z)=S(\phi_1^{|n|/d}-zb'^{|m|/d})+S\phi_2, \quad J_2(z)=S\phi_1+S(\phi_2^{|m|/d}-zc'^{|n|/d}).\eeqs
If $mn<0$, we set
\beqs J_1(z)=S(\phi_1^{|n|/d}-zF'^{|m|/d})+S\phi_2, \quad J_2(z)=S\phi_1+S(\phi_2^{|m|/d}-zE'^{|n|/d}).\eeqs
\begin{prop}The following statements hold.
\item 1. Both $S/I_1$ and $S/I_2$ are domains.
\item 2. $S/I_3$ is simple.
\item 3. Both $S/J_1(z)$ and $S/J_2(z)$ are domains.
\end{prop}
\begin{proof}  1. In the quotient algebra $S/S\phi_1$, we have $E'c'=c'E'$, then $E'c'=q^{2m^2}c'E'+1$ implies $E'c'\in\mk^*$. Consequently,
\beqs S/I_1\cong\mk[E',E'^{-1}][F',b']\cong\mk[c',c'^{-1}][F',b'],\eeqs
which is a domain. In a similar way, $S/I_2$ also is a domain.

2. The quotient algebra $S/I_3$ is isomorphic to the quantum torus $\mk[E'^{\pm1},F'^{\pm1}]$, which is simple.

3. $S/J_1(z)$ is isomorphic to a finite extension of $S_1$. Explicitly,
\beqs S/J_1(z)\cong S_1+S_1\sqrt[|m|]{\phi_1^{|n|}}+S_1\sqrt[|m|]{\phi_1^{2|n|}}+\cdots+S_1\sqrt[|m|]{\phi_1^{(|m|/d-1)|n|}},\eeqs
which is a domain. In special, when $m=\pm n$, $S/J_1(z)\cong S_1$ is simple. The proof for $S/J_2(z)$ is very similar.
\end{proof}

\begin{lemma}
Let $\Omega=\{zb'^ic'^j|z\in\mk^*,i,j\in\mathbb{N}\}$. If $\pp$ is a prime ideal of $S$, then $\Omega\cap\pp=\emptyset$.
\end{lemma}
\begin{proof}
If $zb'^ic'^j\in\pp$, we have $b'^ic'^j\in\pp$. When $i>0$,
\beqs F'b'^ic'^j-q^{-2mnj-2n^2i}b'^ic'^jF'\in\mk^*b'^{i-1}c'^j,\eeqs
thus $b'^{i-1}c'^j\in\pp$. By induction, we infer that $c'^j\in\pp$. Similarly by $E'c'=q^{2m^2}c'E'+1$, we obtain $1\in\pp$ and $\pp=S$, a contradiction.
\end{proof}
\begin{theo}
The prime spectrum of the algebra $S$ is given below,
\beqs{\rm Spec}(S)&=&\{(0), I_1, I_2, I_3\}\cup\{J_1(z)|z\in\mk^*\}\cup\{J_2(z)|z\in\mk^*\}.\eeqs
All the containments of the prime ideals of $S$ are shown in the following diagram
\vskip3mm
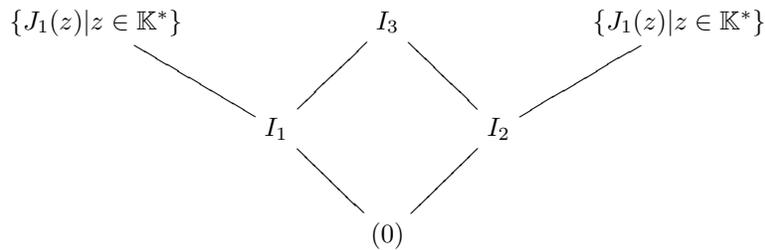
\centerline{
\xymatrix{\{J_1(z)|z\in\mk^*\}\ar@{-}[rd]&&I_3\ar@{-}[rd]\ar@{-}[ld]&&\{J_1(z)|z\in\mk^*\}\ar@{-}[ld]\\
&I_1\ar@{-}[rd]&&I_2\ar@{-}[ld]&\\
&&(0)&&}
}
\end{theo}
\begin{proof}We only prove for $mn>0$.

By Lemma 6.3, the spectrum of $S$ is homeomorphic to the spectrum of its Ore extension $S\Omega^{-1}$. Now we consider the Ore extension
$\mathbb{S}=S\Omega_1^{-1}$ with Ore set
\beqs \Omega_1=\{zb'^ic'^j\phi_1^k\phi_2^l|z\in\mk^*,i,j,k,l\in\mathbb{N}\}.\eeqs
In particular, $\mathbb{S}$ is the localization of $S\Omega^{-1}$ obtained by inverting the elements $\phi_1,\phi_2$.

The algebra $\mathbb{S}$ is a quantum torus with the structural matrix
\beqs Q=\left[\begin{array}{cccc}1&q^{2m^2}&1&1\\q^{-2m^2}&1&q^{-2mn}&1\\1&q^{2mn}&1&q^{2n^2}\\1&1&q^{-2n^2}&1\end{array}\right],\eeqs
associated to the variables $\phi_1,c',b',\phi_2$.

Let $\pp_1$ be a prime ideal of $\mathbb{S}$ and given a nonzero element $x\in\pp_1$ such that
\beqs x=\sum_{i,j,k,l}z_{i,j,k,l\mz}b'^ic'^j\phi_1^k\phi_2^l,\eeqs
with finitely many nonzero coefficients. Because
\beqs \phi_1 x-q^{2m^2j_0}x\phi_1=\sum_{i,j,k,l\in\mz}(1-q^{2m^2(j_0-j)})z_{i,j,k,l}b'^ic'^j\phi_1^k\phi_2^l,\eeqs
we can choose $j_0$ such that $z_{i,j_0,k,l}$ for some $i,k,l$. Either $z_{i,j,k,l}=0$ for all $j\not=j_0$ or we obtain a nonzero element with less nonzero coefficients.

Consider the powers $i,k,l$ in the same way and by induction on the number of nonzero coefficients, we can find a monomial  $zb'^ic'^j\phi_1^k\phi_2^l\in\pp_1$. Therefore $1\in\pp_1$ and it is impossible. This implies $\mathbb{S}$ is simple.

Consequently, for any nonzero prime ideal $\pp$ of $S$ (equivalent to consider the prime ideal $\pp\Omega^{-1}$ of $S\Omega^{-1}$), either $I_1\subseteq\pp$ or $\phi_2\subseteq\pp$.

When $I_3\in\pp$, since $S/I_3$ is isomorphic to a quantum torus, which is simple, we obtain $\pp=I_3$.

Assume $I_1\subseteq\pp\not=I_3$. Because $S\Omega^{-1}/I_1\Omega^{-1}\cong \mk[b',b'^{-1},c',c'^{-1},\phi_2]$,
there exists an irreducible polynomial
\beqs y=\sum_{i,j\in\mz,k\in\mathbb{N}}z_{i,j,k}b'^ic'^j\phi_2^k\in\pp\Omega^{-1}.\eeqs
Similar to above approach, by considering the number of nonzero coefficients of
\beqs\phi_2y-q^{-2n^2i_0}y\phi_2,\eeqs
we may assume $z_{i,j,k}=0$ except $i=i_0$ for some $i_0$. Multiple by $b'^{-i_0}$, we obtain a nonzero irreducible element
\beqs y_1=\sum_{j\in\mz,k\in\mathbb{N}}z_{j,k}c'^j\phi_2^k\in\pp\Omega^{-1},\eeqs
and hence some
\beqs y_2=\sum_{j=0}^rf_{j}(c')\phi_2^j\in\pp,\eeqs
with $r>0$ and $f_0\not=0$. Because
\beqs b'^iy_2=\sum_{j=0}^rq^{2n^2ij}f_{j}(q^{2mni}c')\phi_2^jb'^i,\eeqs
it implies
\beqs \sum_{j=0}^rq^{2n^2ij}f_{j}(q^{2mni}c')\phi_2^j\in\pp,\eeqs
for all $i\in\mathbb{N}.$ Therefore, we can assume
\beqs y_2=\sum_{j=0}^rz_jc'^{i_j}\phi_2^j\in\pp,\eeqs
such that $nj-mi_j$ is a constant. Such irreducible element of the polynomial ring $\mk[c',\phi_2]$ has to be
$\phi_2^{|m|/d}-zc'^{|n|/d}$ for some nonzero constant $z$.  So $J_1(z)\subseteq\pp$. Because $S/J_1(z)$ is simple, we have $\pp=J_1(z)$.

The proof for $I_2\subseteq\pp\not=I_3$ is similar. \end{proof}

\begin{theo}The prime spectrum of the algebra $\od(\frak{b}_{m,n})$ is given below,
\beqs{\rm Spec}(S)=\{(0),\od(\frak{b}_{m,n})^0I_1,\od(\frak{b}_{m,n})^0I_2, \od(\frak{b}_{m,n})^0I_3\}\cup\{\od(\frak{b}_{m,n})^0J_1(z),\od(\frak{b}_{m,n})^0J_2(z)|z\in\mk^*\}.\eeqs
All the containments of the prime ideals of $S$ are shown in the following diagram
\vskip3mm
\centerline{
\xymatrix{\{\od(\frak{b}_{m,n})^0J_1(z)|z\in\mk^*\}\ar@{-}[d]&\od(\frak{b}_{m,n})^0I_3\ar@{-}[rd]\ar@{-}[ld]&\{\od(\frak{b}_{m,n})^0J_1(z)|z\in\mk^*\}\ar@{-}[d]\\
\od(\frak{b}_{m,n})^0I_1\ar@{-}[rd]&&\od(\frak{b}_{m,n})^0I_2\ar@{-}[ld]&\\
&(0)&}
}
\end{theo}
\begin{proof}Because $\od(\frak{b}_{m,n})^0$ is simple, $\od(\frak{b}_{m,n})^0\cap\pp=0$ for any prime ideal of $\od(\frak{b}_{m,n})$.  Assume $\pp$ is a prime ideal of $\od(\frak{b}_{m,n})$. Given a nonzero element
\beqs x=\sum_{i,j\in\mz}x_{i,j}K^ia^j\in\pp\eeqs
for finitely many nonzero elements $x_{i,j}\in S$. Then
\beqs K^ra^sxa^{-s}K^{-r}=\sum_{i,j\in\mz}q^{is-jr}x_{i,j}K^ia^j\in\pp.\eeqs
So it is easy to see that $x_{i,j}\in\pp$ for all $i,j$. Hence $\pp=\od(\frak{b}_{m,n})(\pp\cap S)$. The proof is finished.
\end{proof}

For $m=n=1$, this agrees \cite[Theorem 3.6]{Tao2}.

\section{The automorphism groups}

In this section, we investigate the automorphisms of algebras $\oq(\frak{b}_{m,n})$, $\ou(\frak{b}_{m,n})$ and $\od(\frak{b}_{m,n})$.
\begin{theo}
\item 1. For $m=\pm n$, let $\tau\in{\rm Aut}\oq(\frak{b}_{m,n})$ be defined by
\beqs\tau&=&\left\{\begin{array}{ll}a\mapsto a, b\mapsto c, c\mapsto b,&m=n,\\a\mapsto a^{-1}, b\mapsto c, c\mapsto b,&m=-n.\end{array}\right.\eeqs
Then $\{{\rm id},\tau\}\cong\mz_2$.
\item 2. For any $i\in\mz$,  let $\xi_i\in{\rm Aut}\oq(\frak{b}_{m,n})$ be defined by
\beqs\xi_i: a\mapsto a, b\mapsto a^{\frac{in}{(m,n)}}b, c\mapsto a^{\frac{im}{(m,n)}}c.\eeqs
Then $\{\xi_i|i\in\mz\}\cong\mz$.
\item 3. For all $z,z_1,z_2\in\mk^*$,  let $\zeta_{z,z_1,z_2}\in{\rm Aut}\oq(\frak{b}_{m,n})$ be defined by
\beqs\zeta_{z,z_1,z_2}: a\mapsto za, b\mapsto z_1b, c\mapsto z_2c.\eeqs
Then $\{\zeta_{z,z_1,z_2}|z,z_1,z_2\in\mk^*\}\cong(\mk^*)^3$.
\item4. \beqs{\rm Aut}\oq(\frak{b}_{m,n})&\cong&\left\{\begin{array}{ll}(\mk^*)^3\rtimes(\mz\times\mz_2),&m=\pm n,\\(\mk^*)^3\rtimes\mz,&m\not=\pm n.\end{array}\right.\eeqs
\end{theo}
\begin{proof}It is clear that the statements 1-3 hold. Moreover, the subgroup $G$ generated by $\tau, \xi_i(i\in\mz), \zeta_{z,z_1,z_2}(z,z_1,z_2\in\mk^*)$ is isomorphic to $(\mk^*)^3\rtimes(\mz\times\mz_2)$ and
the subgroup $G_1$ generated by $\xi_i(i\in\mz), \zeta_{z,z_1,z_2}(z,z_1,z_2\in\mk^*)$ is isomorphic to $(\mk^*)^3\rtimes\mz$.

We prove statement 4. Assume $\rho\in {\rm Aut}\oq(\frak{b}_{m,n})$. Obviously, the group $\{za^i|z\in\mk^+,i\in\mz\}=\mk^*\times\langle a\rangle$ is the set of all invertible elements in $\oq(\frak{b}_{m,n})$, which can be generated by
$\mk^*$ and $a$. So this group also can be generated by $\mk^*$ and $\rho(a)$. This implies that $\rho(a)=za^{\pm1}$ for some $z\in\mk^*$.

(i) If $m=n$, $\oq(\frak{b}_{m,n})$ has a decomposition of $a$-weight spaces
\beqs \oq(\frak{b}_{m,n})&=&\bigoplus_{i=0}^\infty \oq(\frak{b}_{m,n})_{q^{im}},\eeqs
under the action $x\mapsto axa^{-1}$, where
\beqs \oq(\frak{b}_{m,n})_{q^{im}}&=&\bigoplus_{r=0}^i\mk[a,a^{-1}]b^{i-r}c^{r}.\eeqs
So either $\rho(b)=xb, \rho(c)=yc$ for some $x,y\in\mk[a,a^{-1}]$ or $\rho(b)=x'c, \rho(c)=y'b$ for some $x',y'\in\mk[a,a^{-1}]$,
and in this case $\rho(a)=za$. Moreover, the elements $x,y,x',y'$ are invertible. By the identity  $\rho(b)\rho(c)=\rho(c)\rho(b)$, we also infer that $xy^{-1},x'y'^{-1}\in\mk^*$. Explicitly, there exists $z,z_1,z_2\in\mk^*$ and $i\in\mz$ such that either
\beqs \rho(a)=za,\;\rho(b)=z_1a^ib,\;\rho(c)=z_2a^{i}c,\eeqs
or
\beqs \rho(a)=za,\;\rho(b)=z_1a^ic,\;\rho(c)=z_2a^{i}b.\eeqs
Then $\rho\in G$, the statement holds.

(ii) If $m=-n$, similar to (i), there exists $z,z_1,z_2\in\mk^*$ and $i\in\mz$ such that either
\beqs \rho(a)=za,\;\rho(b)=z_1a^ib,\;\rho(c)=z_2a^{-i}c,\eeqs
or
\beqs \rho(a)=za^{-1},\;\rho(b)=z_1a^{-i}c,\;\rho(c)=z_2a^{i}b.\eeqs
Then $\rho\in G$, the statement holds.

(iii) If $m\not=\pm n$, we have $\rho(b)\in\mk[a,a^{-1}]b$, $\rho(c)\in\mk[a,a^{-1}]c$. Then similar to above, there exists $z,z_1,z_2\in\mk^*$ and $i\in\mz$ such that either
\beqs \rho(a)=za,\;\rho(b)=z_1a^{\frac{in}{(m,n)}}b,\;\rho(c)=z_2a^{\frac{im}{(m,n)}}c.\eeqs
Then $\rho\in G_1$, the statement holds.
\end{proof}
\begin{theo}
\beqs{\rm Aut}\ou(\frak{b}_{m,n})&\cong&\left\{\begin{array}{ll}(\mk^*)^3\rtimes(\mz\times\mz_2),&m=\pm n,\\(\mk^*)^3\rtimes\mz,&m\not=\pm n.\end{array}\right.\eeqs
\end{theo}
\begin{proof}It is straightforward by the isomorphism of algebras
\beqs \ou (\frak{b}_{m,n})\cong\oq(\frak{b}_{2m,-2n}).\eeqs
\end{proof}

\begin{theo}\label{aut}
\item 1. For any $z_1,z_2\in\mk^*$, let $\zeta_{z_1,z_2}\in {\rm Aut}\od(\frak{b}_{m,n})$ be defined by
\beqs K\mapsto z_1K, a\mapsto z_2a, \rho|_S=\id_S.\eeqs
Then $\{\zeta_{z_1,z_2}|z_1,z_2\in\mk^*\}\cong(\mk^*)^2$.

\item 2. For any $A=(A_{i,j})\in SL_2(\mz)$, let $\rho_A\in {\rm Aut}\od(\frak{b}_{m,n})$ be defined by
\beqs K\mapsto K^{A_{1,1}}a^{A_{2,1}}, a\mapsto K^{A_{1,2}}a^{A_{2,2}},\rho|_S=\id_S.\eeqs
The subgroup generated by $\zeta_{z_1,z_2}(z_1,z_2\in\mk^*), \rho_A, (A\in SL_2(\mz))$ is isomorphic to $(\mk^*)^2\rtimes SL_2(\mz)$.

\item 3. For any $z_3,z_4\in\mk^*$, let $\xi_{z_3,z_4}\in {\rm Aut}\od(\frak{b}_{m,n})$ be defined by
\beqs K\mapsto K, a\mapsto a, E'\mapsto z_3E', F'\mapsto z_4F', c'\mapsto z_3^{-1}c', b\mapsto z_4^{-1}b'.\eeqs
Then $\{\xi_{z_3,z_4}|z_3,z_4\in\mk^*\}\cong(\mk^*)^2$.
\item 4.\beqs{\rm Aut}\od(\frak{b}_{m,n})&\cong&((\mk^*)^2\rtimes SL_2(\mz))\times (\mk^*)^2.\eeqs
\end{theo}
\begin{proof}The statements 1 and 3 are clear.

It is also obvious that the map $\rho_A$ is injective. By the definition, one can verify that $\rho_A\circ\rho_B=\rho_{AB}\circ\zeta_{z_1,z_2}$ for some $z_1,z_2\in\mk^*$. This implies $\rho_A$ is an isomorphism since
$\rho_A\circ\rho_{A^{-1}}\circ\zeta_{z_1^{-1},z_2^{-1}}={\rm id}$. Moreover, the subgroup generated by $\zeta_{z_1,z_2}(z_1,z_2\in\mk^*), \rho_A (A\in SL_2(\mz))$ is isomorphic to $(\mk^*)^2\rtimes SL_2(\mz)$, this proves statement 2.

Obviously, the subgroup $G$ generated by $\zeta_{z_1,z_2}(z_1,z_2\in\mk^*), \rho_A (A\in SL_2(\mz)), \xi_{z_3,z_4}(z_3,z_4\in\mk^*)$ is isomorphic to $((\mk^*)^2\rtimes SL_2(\mz))\times (\mk^*)^2$.

We prove statement 4. Assume $\rho\in {\rm Aut}\oq(\frak{b}_{m,n})$.

Clearly, the group $\{zK^ia^j|z\in\mk^*,i,j\in\mz\}$ is the set of all invertible elements in $\od(\frak{b}_{m,n})$. So $\rho(K)=z_1K^ia^j, \rho(a)=z_2K^ra^s$ for some $z_1,z_2\in\mk^*$ and $i,j,r,s\in\mz$. Moreover,
$\rho(K)\rho(a)=q^{-1}\rho(a)\rho(K)$ implies
\beqs\det\left(\begin{array}{cc}i&r\\j&s\end{array}\right)=1.\eeqs
So $\rho|_{\od(\frak{b}_{m,n})^0}\in {\rm Aut}\od(\frak{b}_{m,n})^0$ and
\beqs \rho|_{\od(\frak{b}_{m,n})^0}=\rho_A\circ\zeta_{z_1,z_2},\; A=\left(\begin{array}{cc}i&r\\j&s\end{array}\right).\eeqs

Since $\rho(\od(\frak{b}_{m,n})^0)=\od(\frak{b}_{m,n})^0$, we also have $\rho(S)=S$ by the fact that  $S$ is the centralizer of $\od(\frak{b}_{m,n})^0$. Next we consider the automorphisms of $S$.

Because $\rho$ maps prime ideals to prime ideals, it must holds
\beqs\{\rho(\phi_1),\rho(\phi_2)\}=\{x\phi_1,y\phi_2\}\eeqs
for some invertible elements $x,y$, so $x,y\in\mk^*$. Up to an automorphism $\zeta_{z_3,z_4}$, we may assume $x=y=1$.

When $\rho(\phi_1)=\phi_1, \rho(\phi_2)=\phi_2$. Because $\rho(c')\rho(\phi_1)=q^{-2m^2}\rho(\phi_1)\rho(c')$ and $\rho(E')\rho(\phi_1)=q^{2m^2}\rho(\phi_1)\rho(E')$, we have
\beqs\rho(c')=f_1(\phi_1,b',F')c',\; \rho(E')=E'f_2(\phi_1,b',F')\eeqs
and
\beqs\mk^*\phi_1+\mk\ni\rho(c')\rho(E')=f_1(\phi_1,b',F')c'E'f_2(\phi_1,b',F'),\eeqs
this implies $f_1(\phi_1,b',F')f_2(\phi_1,b',F')$ is a constant. Moreover, $\rho(c')\rho(E')=c'E'$ forces
\beqs \rho(c')=zc', \rho(E')=z^{-1}E'\eeqs
for some $z\in\mk^*$. Similarly, $\rho(b')=z'c', \rho(E')=z'^{-1}E'$ for some $z'\in\mk^*$.

When $\rho(\phi_1)=\phi_2, \rho(\phi_2)=\phi_1$.  The proof is similar. Hence
\beqs \{\rho|_S \mid \rho\in\od(\frak{b}_{m,n})\}=\{\xi_{z_3,z_4}|_S\mid z_3,z_4\in\mk^*\}={\rm Aut}(S)\cong(\mk^*)^2.\eeqs

The proof is finished.
\end{proof}

\begin{remark}The structure of the automorphism group of $\od(\frak{b}_{m,n})$ is independent on the choice of $m,n$. When $m=n=1$,  Theorem \ref{aut} corrects the result of \cite[Theorem 4.5]{Tao2}, in which  the automorphism group is $(\mk^*)^4\rtimes \mz$.
\end{remark}

\section*{Acknowledgments}
The author gratefully acknowledge partial financial supports from the National Natural Science Foundation of China (11871249, 12171155).

\def\refname{

\end{document}